\documentclass{amsart}
\usepackage{amsfonts}

\newtheorem{theorem}{Theorem}
\theoremstyle{plain}

\newtheorem{definition}{Definition}
\newtheorem{example}{Example}

\newtheorem{lemma}{Lemma}

\newtheorem{proposition}{Proposition}

\numberwithin{equation}{section}
\input{tcilatex}

\begin{document}
\title[On Harmonically Convex Functions]{Hermite-Hadamard type inequalities
for harmonically convex functions}
\author{\.{I}mdat \.{I}\c{s}can}
\address{Department of Mathematics, Faculty of Arts and Sciences,\\
Giresun University, 28100, Giresun, Turkey.}
\email{imdati@yahoo.com}

\begin{abstract}
The author introduce the concept of harmonically convex functions and
establish some Hermite-Hadamard type inequalities of these classes of
functions.
\end{abstract}

\maketitle

\section{Introduction}

Let $f:I\subset \mathbb{R\rightarrow R}$ be a convex function defined on the
interval $I$ of real numbers and $a,b\in I$ with $a<b$. The following
inequality%
\begin{equation}
f\left( \frac{a+b}{2}\right) \leq \frac{1}{b-a}\dint\limits_{a}^{b}f(x)dx%
\leq \frac{f(a)+f(b)}{2}  \label{1-1}
\end{equation}

holds. This double inequality is known in the literature as Hermite-Hadamard
integral inequality for convex functions. Note that some of the classical
inequalities for means can be derived from (\ref{1-1}) for appropriate
particular selections of the mapping $f$. Both inequalities hold in the
reversed direction if $f$ is concave. For some results which generalize,
improve and extend the inequalities(\ref{1-1}) we refer the reader to the
recent papers (see \cite{DP00,I13,S11,SOD10,ZJQ12} ).

The main purpose of this paper is to introduce the concept of harmonically
convex functions and establish some results connected with the right-hand
side of new inequalities similar to the inequality (\ref{1-1}) for these
classes of functions. Some applications to special means of positive real
numbers are also given.

\section{Main Results}

\begin{definition}
Let $I\subset 
\mathbb{R}
\backslash \left\{ 0\right\} $ be an real interval. A function $%
f:I\rightarrow 
\mathbb{R}
$ is said to be harmonically convex, if \ 
\begin{equation}
f\left( \frac{xy}{tx+(1-t)y}\right) \leq tf(y)+(1-t)f(x)  \label{2-1}
\end{equation}%
for all $x,y\in I$ and $t\in \lbrack 0,1]$. If the inequality in (\ref{1-1})
is reversed, then $f$ is said to be harmonically concave.
\end{definition}

\begin{example}
Let $f:\left( 0,\infty \right) \rightarrow 
\mathbb{R}
,\ f(x)=x,$ and $g:\left( -\infty ,0\right) \rightarrow 
\mathbb{R}
,\ g(x)=x,$ then $f$ is a harmonically convex function and $g$ is a
harmonically concave function.
\end{example}

The following proposition is obvious from this example:

\begin{proposition}
Let $I\subset 
\mathbb{R}
\backslash \left\{ 0\right\} $ be an real interval and $f:I\rightarrow 
\mathbb{R}
$ is a function, then ;

\begin{itemize}
\item if $I\subset \left( 0,\infty \right) $ and f is convex and
nondecreasing function then f is harmonically convex.

\item if $I\subset \left( 0,\infty \right) $ and f is harmonically convex
and nonincreasing function then f is convex.

\item if $I\subset \left( -\infty ,0\right) $ and f is harmonically convex
and nondecreasing function then f is convex.

\item if $I\subset \left( -\infty ,0\right) $ and f is convex and
nonincreasing function then f is a harmonically convex.
\end{itemize}
\end{proposition}

The following result of the Hermite-Hadamard type holds.

\begin{theorem}
\label{2.2}Let $f:I\subset 
\mathbb{R}
\backslash \left\{ 0\right\} \rightarrow 
\mathbb{R}
$ be an harmonically convex function and $a,b\in I$ with $a<b.$ If $f\in
L[a,b]$ then the following inequalities hold%
\begin{equation}
f\left( \frac{2ab}{a+b}\right) \leq \frac{ab}{b-a}\dint\limits_{a}^{b}\frac{%
f(x)}{x^{2}}dx\leq \frac{f(a)+f(b)}{2}.  \label{2-2}
\end{equation}%
The \ above inequalities are sharp.
\end{theorem}

\begin{proof}
Since $f:I\rightarrow 
\mathbb{R}
$ be an harmonically convex, we have, for all $x,y\in I$ (with $t=\frac{1}{2}
$ in the inequality (\ref{2-1}) )%
\begin{equation*}
f\left( \frac{2xy}{x+y}\right) \leq \frac{f(y)+f(x)}{2}
\end{equation*}%
Choosing $x=\frac{ab}{ta+(1-t)b},\ y=\frac{ab}{tb+(1-t)a}$, we get%
\begin{equation*}
f\left( \frac{2ab}{a+b}\right) \leq \frac{f\left( \frac{ab}{tb+(1-t)a}%
\right) +f\left( \frac{ab}{ta+(1-t)b}\right) }{2}
\end{equation*}%
Further, integrating for $t\in \lbrack 0,1]$, we have%
\begin{equation}
f\left( \frac{2ab}{a+b}\right) \leq \frac{1}{2}\left[ \dint\limits_{0}^{1}f%
\left( \frac{ab}{tb+(1-t)a}\right) dt+\dint\limits_{0}^{1}f\left( \frac{ab}{%
ta+(1-t)b}\right) dt\right]  \label{2-3}
\end{equation}%
Since each of the integrals is equal to $\frac{ab}{b-a}\dint\limits_{a}^{b}%
\frac{f(x)}{x^{2}}dx$, we obtain the inequality (\ref{2-2}) from (\ref{2-3}).

The proof of the second inequality follows by using (\ref{2-1}) with $x=a$
and $y=b$ and integrating with respect to $t$ over $[0,1]$.

Now, consider the function $f:\left( 0,\infty \right) \rightarrow 
\mathbb{R}
,$ $f(x)=1.$ thus%
\begin{eqnarray*}
1 &=&f\left( \frac{xy}{tx+(1-t)y}\right) \\
&=&tf(y)+(1-t)f(x)=1
\end{eqnarray*}%
for all $x,y\in \left( 0,\infty \right) $ and $t\in \lbrack 0,1].$ Therefore 
$f$ is harmonically convex on $\left( 0,\infty \right) .$ We also have%
\begin{equation*}
f\left( \frac{2ab}{a+b}\right) =1,\ \frac{ab}{b-a}\dint\limits_{a}^{b}\frac{%
f(x)}{x^{2}}dx=1,\ \text{and }\frac{f(a)+f(b)}{2}=1
\end{equation*}%
which shows us the inequalities (\ref{2-2}) are sharp.
\end{proof}

For finding some new inequalities of Hermite-Hadamard type for functions
whose derivatives are harmonically convex, we need a simple lemma below.

\begin{lemma}
\label{2.1}Let $f:I\subset 
\mathbb{R}
\backslash \left\{ 0\right\} \rightarrow 
\mathbb{R}
$ be a differentiable function on $I^{\circ }$ and $a,b\in I$ with $a<b$. If 
$f^{\prime }\in L[a,b]$ then 
\begin{eqnarray}
&&\frac{f(a)+f(b)}{2}-\frac{ab}{b-a}\dint\limits_{a}^{b}\frac{f(x)}{x^{2}}dx
\notag \\
&=&\frac{ab\left( b-a\right) }{2}\dint\limits_{0}^{1}\frac{1-2t}{\left(
tb+(1-t)a\right) ^{2}}f^{\prime }\left( \frac{ab}{tb+(1-t)a}\right) dt
\label{2-4}
\end{eqnarray}
\end{lemma}

\begin{proof}
Let%
\begin{equation*}
I^{\ast }=\frac{ab\left( b-a\right) }{2}\dint\limits_{0}^{1}\frac{1-2t}{%
\left( tb+(1-t)a\right) ^{2}}f^{\prime }\left( \frac{ab}{tb+(1-t)a}\right)
dt.
\end{equation*}%
By integrating by part, we have%
\begin{equation*}
I^{\ast }=\left. \frac{\left( 2t-1\right) }{2}f\left( \frac{ab}{tb+(1-t)a}%
\right) \right\vert _{0}^{1}-\dint\limits_{0}^{1}f\left( \frac{ab}{tb+(1-t)a}%
\right) dt
\end{equation*}%
Setting $x=\frac{ab}{tb+(1-t)a},$ $dx=\frac{-ab(b-a)}{\left(
tb+(1-t)a\right) ^{2}}dt=\frac{-x^{2}(b-a)}{ab}dt$, we obtain%
\begin{equation*}
I^{\ast }=\frac{f(a)+f(b)}{2}-\frac{ab}{b-a}\dint\limits_{a}^{b}\frac{f(x)}{%
x^{2}}dx
\end{equation*}%
which gives the desired representation (\ref{2-4}).
\end{proof}

\begin{theorem}
Let $f:I\subset \left( 0,\infty \right) \rightarrow 
\mathbb{R}
$ be a differentiable function on $I^{\circ }$, $a,b\in I$ with $a<b,$ and $%
f^{\prime }\in L[a,b].$ If $\left\vert f^{\prime }\right\vert ^{q}$ is
harmonically convex on $[a,b]$ for $q\geq 1,$ then%
\begin{eqnarray}
&&\left\vert \frac{f(a)+f(b)}{2}-\frac{ab}{b-a}\dint\limits_{a}^{b}\frac{f(x)%
}{x^{2}}dx\right\vert  \label{2-5} \\
&\leq &\frac{ab\left( b-a\right) }{2}\lambda _{1}^{1-\frac{1}{q}}\left[
\lambda _{2}\left\vert f^{\prime }\left( a\right) \right\vert ^{q}+\lambda
_{3}\left\vert f^{\prime }\left( b\right) \right\vert ^{q}\right] ^{\frac{1}{%
q}},  \notag
\end{eqnarray}%
where 
\begin{eqnarray*}
\lambda _{1} &=&\frac{1}{ab}-\frac{2}{\left( b-a\right) ^{2}}\ln \left( 
\frac{\left( a+b\right) ^{2}}{4ab}\right) , \\
\lambda _{2} &=&\frac{-1}{b\left( b-a\right) }+\frac{3a+b}{\left( b-a\right)
^{3}}\ln \left( \frac{\left( a+b\right) ^{2}}{4ab}\right) , \\
\lambda _{3} &=&\frac{1}{a\left( b-a\right) }-\frac{3b+a}{\left( b-a\right)
^{3}}\ln \left( \frac{\left( a+b\right) ^{2}}{4ab}\right) \\
&=&\lambda _{1}-\lambda _{2}
\end{eqnarray*}
\end{theorem}

\begin{proof}
From Lemma \ref{2.1} and using the power mean inequality, we have%
\begin{eqnarray*}
&&\left\vert \frac{f(a)+f(b)}{2}-\frac{ab}{b-a}\dint\limits_{a}^{b}\frac{f(x)%
}{x^{2}}dx\right\vert \\
&\leq &\frac{ab\left( b-a\right) }{2}\dint\limits_{0}^{1}\left\vert \frac{%
1-2t}{\left( tb+(1-t)a\right) ^{2}}\right\vert \left\vert f^{\prime }\left( 
\frac{ab}{tb+(1-t)a}\right) \right\vert dt \\
&\leq &\frac{ab\left( b-a\right) }{2}\left( \dint\limits_{0}^{1}\left\vert 
\frac{1-2t}{\left( tb+(1-t)a\right) ^{2}}\right\vert dt\right) ^{1-\frac{1}{q%
}} \\
&&\times \left( \dint\limits_{0}^{1}\left\vert \frac{1-2t}{\left(
tb+(1-t)a\right) ^{2}}\right\vert \left\vert f^{\prime }\left( \frac{ab}{%
tb+(1-t)a}\right) \right\vert ^{q}dt\right) ^{\frac{1}{q}}.
\end{eqnarray*}%
Hence, by harmonically convexity of $\left\vert f^{\prime }\right\vert ^{q}$
on $[a,b],$we have%
\begin{eqnarray*}
&&\left\vert \frac{f(a)+f(b)}{2}-\frac{ab}{b-a}\dint\limits_{a}^{b}\frac{f(x)%
}{x^{2}}dx\right\vert \\
&\leq &\frac{ab\left( b-a\right) }{2}\left( \dint\limits_{0}^{1}\frac{%
\left\vert 1-2t\right\vert }{\left( tb+(1-t)a\right) ^{2}}dt\right) ^{1-%
\frac{1}{q}} \\
&&\times \left( \dint\limits_{0}^{1}\frac{\left\vert 1-2t\right\vert }{%
\left( tb+(1-t)a\right) ^{2}}\left[ t\left\vert f^{\prime }\left( a\right)
\right\vert ^{q}+(1-t)\left\vert f^{\prime }\left( b\right) \right\vert ^{q}%
\right] dt\right) ^{\frac{1}{q}} \\
&\leq &\frac{ab\left( b-a\right) }{2}\lambda _{1}^{1-\frac{1}{q}}\left[
\lambda _{2}\left\vert f^{\prime }\left( a\right) \right\vert ^{q}+\lambda
_{3}\left\vert f^{\prime }\left( b\right) \right\vert ^{q}\right] ^{\frac{1}{%
q}}.
\end{eqnarray*}%
It is easily check that%
\begin{equation*}
\dint\limits_{0}^{1}\frac{\left\vert 1-2t\right\vert }{\left(
tb+(1-t)a\right) ^{2}}dt=\frac{1}{ab}-\frac{2}{\left( b-a\right) ^{2}}\ln
\left( \frac{\left( a+b\right) ^{2}}{4ab}\right) ,
\end{equation*}%
\begin{equation*}
\dint\limits_{0}^{1}\frac{\left\vert 1-2t\right\vert \left( 1-t\right) }{%
\left( tb+(1-t)a\right) ^{2}}dt=\frac{1}{a\left( b-a\right) }-\frac{3b+a}{%
\left( b-a\right) ^{3}}\ln \left( \frac{\left( a+b\right) ^{2}}{4ab}\right) ,
\end{equation*}%
\begin{equation*}
\dint\limits_{0}^{1}\frac{\left\vert 1-2t\right\vert t}{\left(
tb+(1-t)a\right) ^{2}}dt=\frac{-1}{b\left( b-a\right) }+\frac{3a+b}{\left(
b-a\right) ^{3}}\ln \left( \frac{\left( a+b\right) ^{2}}{4ab}\right) .
\end{equation*}
\end{proof}

\begin{theorem}
Let $f:I\subset \left( 0,\infty \right) \rightarrow 
\mathbb{R}
$ be a differentiable function on $I^{\circ }$, $a,b\in I$ with $a<b,$ and $%
f^{\prime }\in L[a,b].$ If $\left\vert f^{\prime }\right\vert ^{q}$ is
harmonically convex on $[a,b]$ for $q>1,\;\frac{1}{p}+\frac{1}{q}=1,$ then%
\begin{eqnarray}
&&\left\vert \frac{f(a)+f(b)}{2}-\frac{ab}{b-a}\dint\limits_{a}^{b}\frac{f(x)%
}{x^{2}}dx\right\vert  \label{2-6} \\
&\leq &\frac{ab\left( b-a\right) }{2}\left( \frac{1}{p+1}\right) ^{\frac{1}{p%
}}\left( \mu _{1}\left\vert f^{\prime }\left( a\right) \right\vert ^{q}+\mu
_{2}\left\vert f^{\prime }\left( b\right) \right\vert ^{q}\right) ^{\frac{1}{%
q}},  \notag
\end{eqnarray}%
where%
\begin{eqnarray*}
\mu _{1} &=&\frac{1}{2\left( b-a\right) ^{2}\left( 1-q\right) \left(
1-2q\right) }\left[ a^{2-2q}+b^{1-2q}\left[ \left( b-a\right) \left(
1-2q\right) -a\right] \right] , \\
\mu _{2} &=&\frac{1}{2\left( b-a\right) ^{2}\left( 1-q\right) \left(
1-2q\right) }\left[ b^{2-2q}-a^{1-2q}\left[ \left( b-a\right) \left(
1-2q\right) +b\right] \right] .
\end{eqnarray*}
\end{theorem}

\begin{proof}
From Lemma \ref{2.1}, H\"{o}lder's inequality and the harmonically convexity
of $\left\vert f^{\prime }\right\vert ^{q}$ on $[a,b],$we have, 
\begin{eqnarray*}
\left\vert \frac{f(a)+f(b)}{2}-\frac{ab}{b-a}\dint\limits_{a}^{b}\frac{f(x)}{%
x^{2}}dx\right\vert &\leq &\frac{ab\left( b-a\right) }{2}\left(
\dint\limits_{0}^{1}\left\vert 1-2t\right\vert ^{p}dt\right) ^{\frac{1}{p}}
\\
&&\times \left( \dint\limits_{0}^{1}\frac{1}{\left( tb+(1-t)a\right) ^{2q}}%
\left\vert f^{\prime }\left( \frac{ab}{tb+(1-t)a}\right) \right\vert
^{q}dt\right) ^{\frac{1}{q}} \\
&\leq &\frac{ab\left( b-a\right) }{2}\left( \frac{1}{p+1}\right) ^{\frac{1}{p%
}} \\
&&\times \left( \dint\limits_{0}^{1}\frac{t\left\vert f^{\prime }\left(
a\right) \right\vert ^{q}+(1-t)\left\vert f^{\prime }\left( b\right)
\right\vert ^{q}}{\left( tb+(1-t)a\right) ^{2q}}dt\right) ^{\frac{1}{q}},
\end{eqnarray*}%
where an easy calculation gives%
\begin{eqnarray}
&&\dint\limits_{0}^{1}\frac{t}{\left( tb+(1-t)a\right) ^{2q}}dt  \label{2-7}
\\
&=&\frac{1}{2\left( b-a\right) ^{2}\left( 1-q\right) \left( 1-2q\right) }%
\left[ a^{2-2q}+b^{1-2q}\left[ \left( b-a\right) \left( 1-2q\right) -a\right]
\right]  \notag
\end{eqnarray}%
and%
\begin{eqnarray}
&&\dint\limits_{0}^{1}\frac{1-t}{\left( tb+(1-t)a\right) ^{2q}}dt
\label{2-8} \\
&=&\frac{1}{2\left( b-a\right) ^{2}\left( 1-q\right) \left( 1-2q\right) }%
\left[ b^{2-2q}-a^{1-2q}\left[ \left( b-a\right) \left( 1-2q\right) +b\right]
\right]  \notag
\end{eqnarray}%
Substituting equations (\ref{2-7}) and (\ref{2-8}) into the above inequality
results in the inequality (\ref{2-6}), which completes the proof.
\end{proof}

\section{Some applications for special means}

Let us recall the following special means of two nonnegative number $a,b$
with $b>a:$

\begin{enumerate}
\item The arithmetic mean%
\begin{equation*}
A=A\left( a,b\right) :=\frac{a+b}{2}.
\end{equation*}

\item The harmonic mean%
\begin{equation*}
H=H\left( a,b\right) :=\frac{2ab}{a+b}.
\end{equation*}

\item The Logarithmic mean%
\begin{equation*}
L=L\left( a,b\right) :=\frac{b-a}{\ln b-\ln a}.
\end{equation*}

\item The p-Logarithmic mean%
\begin{equation*}
L_{p}=L_{p}\left( a,b\right) :=\left( \frac{b^{p+1}-a^{p+1}}{(p+1)(b-a)}%
\right) ^{\frac{1}{p}},\ \ p\in 
\mathbb{R}
\backslash \left\{ -1,0\right\} .
\end{equation*}

\item the Identric mean%
\begin{equation*}
I=I\left( a,b\right) =\frac{1}{e}\left( \frac{b^{b}}{a^{a}}\right) ^{\frac{1%
}{b-a}}.
\end{equation*}
\end{enumerate}

These means are often used in numerical approximation and in other areas.
However, the following simple relationships are known in the literature:%
\begin{equation*}
H\leq G\leq L\leq I\leq A.
\end{equation*}%
It is also known that $L_{p}$ is monotonically increasing over $p\in 
\mathbb{R}
,$ denoting $L_{0}=I$ and $L_{-1}=L.$

\begin{proposition}
Let $0<a<b.$ Then we have the following inequality%
\begin{equation*}
H\leq \frac{G^{2}}{L}\leq A.
\end{equation*}
\end{proposition}

\begin{proof}
The assertion follows from the inequality (\ref{2-2}) in Theorem \ref{2.2},
for $f:\left( 0,\infty \right) \rightarrow 
\mathbb{R}
,\ f(x)=x.$
\end{proof}

\begin{proposition}
Let $0<a<b.$ Then we have the following inequality%
\begin{equation*}
H^{2}\leq G^{2}\leq A(a^{2},b^{2}).
\end{equation*}
\end{proposition}

\begin{proof}
The assertion follows from the inequality (\ref{2-2}) in Theorem \ref{2.2},
for $f:\left( 0,\infty \right) \rightarrow 
\mathbb{R}
,\ f(x)=x^{2}.$
\end{proof}

\begin{proposition}
Let $0<a<b$ and $p\in \left( -1,\infty \right) \backslash \left\{ 0\right\}
. $ Then we have the following inequality%
\begin{equation*}
H^{p+2}\leq G^{2}.L_{p}^{p}\leq A(a^{p+2},b^{p+2}).
\end{equation*}
\end{proposition}

\begin{proof}
The assertion follows from the inequality (\ref{2-2}) in Theorem \ref{2.2},
for $f:\left( 0,\infty \right) \rightarrow 
\mathbb{R}
,\ f(x)=x^{p+2},\ p\left( -1,\infty \right) \backslash \left\{ 0\right\} .$
\end{proof}

\begin{proposition}
Let $0<a<b.$ Then we have the following inequality%
\begin{equation*}
H^{2}\ln H\leq G^{2}\ln I\leq A\left( a^{2}\ln a,b^{2}\ln b\right) .
\end{equation*}
\end{proposition}

\begin{proof}
The assertion follows from the inequality (\ref{2-2}) in Theorem \ref{2.2},
for $f:\left( 0,\infty \right) \rightarrow 
\mathbb{R}
,\ f(x)=x^{2}\ln x.$
\end{proof}

\end{document}